\documentclass[10pt]{article}
\usepackage{mathrsfs}
\usepackage{amsmath,amscd,amsthm,amsfonts}
\usepackage{graphicx}
\oddsidemargin=0.5cm\topmargin=-1cm \numberwithin{equation}{section}

\theoremstyle{plain}
\newtheorem{theorem}{Theorem}[section]

\newtheorem{definition}[theorem]{Definition}
\newtheorem{lemma}[theorem]{Lemma}

\newtheorem{proposition}[theorem]{Proposition}
\newtheorem{remark}[theorem]{Remark}

\newtheorem*{CCL}{Center Closing Lemma}
\newtheorem*{CSD}{Center Spectral Decomposition}
\newtheorem*{QSP}{Quasi-shadowing Property}

\begin{document}
\vskip 5.1cm
\title{Center Specification Property and Entropy for Partially Hyperbolic Diffeomorphisms
\footnotetext{
It is supported by NSFC(No: 11371120). The second author is also supported by NCET(No: 11-0935), NSFHB(No: A2014205154, A2013205148), BRHB(No: BR2-219) and GCCHB(No: GCC2014052)}}
\author {Lin Wang and Yujun Zhu \\
\small {College of Mathematics and Information Science, and}\\
\small {Hebei Key Laboratory of Computational Mathematics and Applications,}\\
\small {Hebei Normal University, Shijiazhuang, 050024, P.R.China}}
\date{}
\maketitle

\begin{abstract}
Let $f$ be a partially hyperbolic diffeomorphism on a closed (i.e., compact and boundaryless) Riemannian manifold $M$ with a uniformly compact center foliation $\mathcal{W}^{c}$. The relationship among topological entropy $h(f)$, entropy of the restriction of $f$ on the center foliation $h(f, \mathcal{W}^{c})$ and the growth rate of periodic center leaves $p^{c}(f)$ is investigated. It is first shown that if a compact locally maximal invariant center set $\Lambda$ is center topologically mixing  then $f|_{\Lambda}$ has the center specification property, i.e., any specification with a large spacing can be center shadowed by a periodic center leaf  with a fine precision. Applying the center spectral decomposition and the center specification property, we show that
 $ h(f)\leq h(f,\mathcal{W}^{c})+p^{c}(f)$. Moreover, if the center foliation $\mathcal{W}^{c}$ is of dimension one, we obtain an equality $h(f)= p^{c}(f)$.
\end{abstract}
\section{Introduction}

The investigation of complexity of the orbit structure is one of the main topics in dynamical systems. It is well known that topological entropy is the most important numerical invariant related to the total exponential complexity of the orbit structure. It represents the exponential growth rate for the number of orbit segments distinguishable with arbitrarily fine but finite precision.

In smooth dynamical systems, the hyperbolicity implies the coincidence of topological entropy and growth rate of the particular type of orbits: periodic orbits, that is, for a diffeomorphism $f$ on a closed Riemannian manifold $M$, if $\Lambda$ is a compact locally maximal hyperbolic set for $f$, then
\begin{equation}\label{hyper1}
h(f|_{\Lambda})=p(f|_{\Lambda}),
\end{equation}
 where $h(f|_{\Lambda})$ and $p(f|_{\Lambda})$ are the topological entropy and the growth rate of periodic orbits of $f|_{\Lambda}$ respectively. Therefore, if $f$ is uniformly hyperbolic then
\begin{equation}\label{hyper}
h(f)=p(f)
\end{equation}
by the spectral decomposition and the specification property (see section 18.5 of \cite{Katok}, for example).

The main goal of this paper is to obtain analogues of the above equalities (\ref{hyper1}) and (\ref{hyper}) for partially hyperbolic diffeomorphisms. Roughly speaking, a diffeomorphism $f : M\longrightarrow M$
is called \emph{partially hyperbolic} if the tangent bundle $TM$ is decomposed into three non-trivial,
$df$-invariant subbundles, called the \emph{stable}, \emph{unstable} and \emph{center} bundle, such
that $df$ contracts uniformly vectors in the stable direction, expands uniformly
vectors in the unstable direction and contracts and/or expands in a weaker
way vectors in the center direction. The presence of the center direction permits a rich type of structure and induces more complicated dynamics. Partial hyperbolicity theory was introduced by Brin and
Pesin (\cite{Brin}) and independently by Hirsh, Pugh and Shub \cite{Hirsch}. Some works that appeared in the nineties opened the way for making partial hyperbolicity one of the most active topics in dynamics
in recent years.  For general theory of partially hyperbolic system and recent progress, we refer to the books \cite{Hirsch}, \cite{Pesin}, \cite{Barreira}, \cite{Bonatti} and \cite{Hertz} and the papers in their references.

Generally, for a partially hyperbolic diffeomorphism $f$, the stable and unstable foliations always exist, while the center bundle might not be integrable. In this paper, we restrict ourselves to the case that the center foliation exists and is uniformly compact. For some recent results on dynamics of partially hyperbolic diffeomorphisms with a uniformly compact center foliation, we refer to \cite{Bohnet}, \cite{Bonatti1}, \cite{Carrasco11} and \cite{Carrasco}.

The counterpart of $p(f)$ in (\ref{hyper}) is  $p^c(f)$, i.e., the growth rate of periodic center leaves, for the partially hyperbolic case. However, we can not expect the equality $h(f)=p^c(f)$ holds  since the center direction may has contribution to the entropy. So we should pay more attention to the relationship among topological entropy, entropy of the center foliation which is denoted by $h(f,\mathcal{W}^{c})$ and the growth rate of periodic center leaves.

As we mentioned above, to get (\ref{hyper}) for uniformly hyperbolic systems one always resorts to the specification property and the spectral decomposition theorem.  For the partially hyperbolic diffeomorphisms, we will first give a center specification property applying the techniques and results in \cite{Hu} and \cite{Hu1}.

\begin{theorem}\label{centerspecification}
Let $f$ be a partially hyperbolic diffeomorphism on $M$ with a uniformly compact center foliation and $\Lambda$ be a compact locally maximal invariant center set of $f$. If $\Lambda$ is center
topologically mixing, then $f|_{\Lambda}$ has the center specification property.
\end{theorem}

Applying Theorem \ref{centerspecification} and the center spectral decomposition theorem in \cite{Hu} to  investigate the relation among $h(f)$, $p^c(f)$ and $h(f,\mathcal{W}^{c})$, we get the following result.

\begin{theorem}\label{entropyrelation}
Let $f:M\rightarrow M$ be a partially hyperbolic diffeomorphism with uniformly compact center foliation. Then
\begin{equation}\label{entropyinequality}
h(f)\leq p^{c}(f)+h(f, \mathcal{W}^{c}).
\end{equation}
Moreover, if the center foliation $\mathcal{W}^{c}$ is of dimension one, then
\begin{equation}\label{entropyequality}
h(f)=p^{c}(f).
\end{equation}
\end{theorem}

\begin{remark}
Generally, the equality (\ref{entropyequality}) does not hold. For example, let $f_{1}$ be an Anosov diffeomorphism on $M$ and $f_{2}$ a diffeomorphism on $N$ which has a horseshoe. We can choose $f_{1}$ and $f_{2}$ appropriately such that the direct product $f=f_{1}\times f_{2}$ is a partially hyperbolic diffeomorphism on $M\times N$. Clearly for any $x=(x_1,x_2)\in M\times N$, its center leaf $W^c(x)$ is $\{x_1\}\times N$. Note that $h(f)=h(f_{1})+h(f_{2}),\;h(f_{1})=p(f_{1})=p^{c}(f)$, and $h(f_{2})>0$ since $f_{2}$ has a horseshoe. Therefore $h(f)>p^{c}(f)$.
\end{remark}

This paper is organized as follows.
The statements of preliminaries about a partially hyperbolic diffeomorphism $f$ are given in Section 2.  In Section 3, we apply the center closing lemma to establish a center specification property (Theorem \ref{centerspecification}) for any compact locally maximal invariant center set. Section 4 is devoted to the  proof of  Theorem \ref{entropyrelation}.

\section{Preliminaries}

Everywhere in this paper, we assume that $M$ is a smooth
 closed Riemannian manifold. We denote by $\|\cdot\|$
and $d(\cdot,\cdot)$ the norm on $TM$ and the distance function on $M$ induced by the Riemannian metric, respectively.

A diffeomorphism $f: M\to M$ is said to be
\emph{(absolutely} or \emph{uniformly) partially hyperbolic} if there exist
numbers $\lambda,\lambda',\mu$ and $\mu'$ with
$0<\lambda<1<\mu$ and $\lambda<\lambda'\leq \mu'<\mu$,
and an invariant decomposition $TM=E^s\oplus E^c\oplus E^u$
such that for any $n\ge 0$,
\begin{eqnarray*}
\|d_xf^nv\|  &\!\!\!\!\!  \le C\lambda^n\|v\|\ \   &\text{as} \ v\in E^s(x), \\
C^{-1}(\lambda')^n\|v\| \le \|d_xf^nv\|  &\!\!\!  \le C(\mu')^n\|v\|
               \ &\text{as} \   v\in E^c(x),\\
C^{-1}\mu^n\,\|v\| \leq  \|d_xf^nv\|&\ &\text{as} \  v\in E^u(x)
\end{eqnarray*}
hold for some number $C>0$.  The subspaces
$E_x^s, E_x^c$ and $E_x^u$ are called
\emph{stable, center} and \emph{unstable} subspace, respectively.
Via a change of Riemannian metric we always assume that $C=1$. Moreover, for simplicity of the notation, we assume that $\lambda=1/\mu$ $($otherwise, if $\lambda<1/\mu$, we can choose $\lambda^{'}\in(\lambda,1)$ satisfying $\lambda^{'}=1/\mu$, or if $\lambda>1/\mu$, we can choose $\mu^{'}\in(1,\mu)$ satisfying $\lambda=1/\mu^{'}$$)$. Generally, $E^u$ and $E^s$ are
integrable, and everywhere tangent to $\mathcal{W}^u$
and $\mathcal{W}^s$, called the \emph{unstable} and \emph{stable foliations}, respectively. If $E^c$ is
integrable, then the induced foliation $\mathcal{W}^c$  is called the \emph{center foliation}. For any $x\in M$, we denote the \emph{unstable, stable} and \emph{center manifolds} (also called \emph{leaves}) at $x$ by $W^u(x)$,  $W^s(x)$  and $W^c(x)$ respectively, and for $\varepsilon>0$, denote by $W^u_\varepsilon(x)$,  $W^s_\varepsilon(x)$  and $W^c_\varepsilon(x)$ the corresponding local leaves of size $\varepsilon$, i.e.,
$$W^{u}_{\varepsilon}(x)=\{y\in W^{u}(x)|d(f^{-n}(x),f^{-n}(y))<\varepsilon,\;n\geq0\},$$
$$W^{s}_{\varepsilon}(x)=\{y\in W^{s}(x)|d(f^{n}(x),f^{n}(y))<\varepsilon,\;n\geq0\},$$
$$W^{c}_{\varepsilon}(x)=\{y\in W^{c}(x)|d(f^{n}(x),f^{n}(y))<\varepsilon,\;n\in\mathbb{Z}\}.$$

For a partially hyperbolic diffeomorphism $f$ on $M$, there are two important properties: ``dynamical coherence" and ``plaque expansivity" on the foliations, which play important roles in the investigation of the dynamics of the system. The map $f$ is called \emph{dynamically coherent} if $E^{cu}:=E^c\oplus E^u$, $E^{c}$, and $E^{cs}:=E^c\oplus E^s$ are
integrable, and everywhere tangent to $\mathcal{W}^{cu}$ (which is called the \emph{center-unstable foliation}),
$\mathcal{W}^{c}$ and $\mathcal{W}^{cs}$  (which is called the \emph{center-stable foliation}),
respectively; and $\mathcal{W}^c$ and $\mathcal{W}^u$ are
subfoliations of $\mathcal{W}^{cu}$, while $\mathcal{W}^c$ and
$\mathcal{W}^s$ are subfoliations of $\mathcal{W}^{cs}$.
The map $f$ is called \emph{plaque expansive} with respect to the center foliation  $\mathcal{W}^c$ if there exists $\eta>0$ such that for any $\eta$-pseudo orbits $\{x_n\}_{n=-\infty}^{\infty}$ and $\{y_n\}_{n=-\infty}^{\infty}$ in which $f(x_n)$ and $f(y_n)$  lie in the center plaque $W^c_{\eta}(x_{n+1})$ and $W^c_{\eta}(y_{n+1})$ respectively, if $d(x_n, y_n)<\eta$ then $x_n$ and $y_n$ must lie in a common center plaque, i.e., $x_n\in W^c_{\eta}(y_n)$. Such an $\eta$ is called a \emph{plaque expansiveness constant} of $f$. We assume that $f$ has a compact center foliation $\mathcal{W}^c$, that is, $W^{c}(x)$ is compact for each $x$.
We say that the compact center foliation $\mathcal{W}^c$ is \emph{uniformly compact} if
$$
\sup\{\mbox{vol}(W^{c}(x)): x\in M\}<+\infty,
 $$
 where $\mbox{vol}(W^{c}(x))$ is the Riemannian volume restricted to the submanifold $W^{c}(x)$ of $M$. If the center foliation $\mathcal{W}^c$ is uniformly compact, then by Theorem 1 of \cite{Bonatti1}, the map $f$ is dynamically coherent, and by Proposition 13 of \cite{Pugh}, the map $f$ is plaque expansive.

In the remaining of this paper, we always assume that  the partially hyperbolic diffeomorphism $f$ has a uniformly compact center foliation except we declare in advance, and hence $f$ is dynamically coherent and plaque expansive. By \cite{Hu} and \cite{Hu1}, the quasi-shadowing property (without the assumption of uniform compactness on the center foliation), the center closing lemma and the center spectral decomposition hold for $f$. We can also see some similar results for quasi-shadowing property (or say center shadowing property)  in \cite{Bonatti1}, \cite{Carrasco11} and \cite{Tikhomirov}. Here we state them in the following versions.

\begin{QSP}[\cite{Hu1}]\label{QSP}
Let $f$ be a dynamically coherent partially hyperbolic diffeomorphism on $M$.
Then for any $\varepsilon>0,$ there exists $\delta\in(0,\varepsilon)$ such that for any $\delta$-pseudo orbit  $\xi=\{x_{k}\}_{-\infty}^{+\infty}$ with $\sup_{k\in \mathbb{Z}} d(f(x_{k}),\;x_{k+1})\leq \delta$
there is a sequence
$\{y_k\}_{k\in \mathbb{Z}}$ with $y_k\in W^c_{\varepsilon}(f(y_{k-1}))$ \emph{$\varepsilon$-shadowing} $\xi$ in the following sense:
$$
\max\{d(x_k,y_k):k\in \mathbb{Z}\}<\varepsilon.
 $$
 Moreover, if the center foliation is uniformly compact then reduce $\delta$ if necessary we can obtain
 $$
\max\{d_H(W^c(x_k), W^c(y_k)):k\in \mathbb{Z}\}<\varepsilon,
 $$
 where $d_{H}(\cdot,\cdot)$ denotes the Hausdorff distance given by
 $$
 d_{H}(A,B)=\max_{a\in A}\min_{b\in B}d(a,b)
 $$
 for closed subsets $A,B\subset M$. Note that
 $$
 d_H(W^c(x_k), W^c(y_k))=d_H(W^c(x_k), W^c(f^k(y_0))),
  $$
  so in this case we call the point $y_0$ \emph{center $\varepsilon$-shadows} the pseudo orbit $\xi$.
\end{QSP}

A center leaf $W^{c}(p)$ is said to be a \emph{periodic center leaf} with period $n\in\mathbb{N}$ if $W^{c}(p)=W^{c}(f^{n}(p)).$
Denote
$$
P^{c}(f)=\{p\in M : W^{c}(p)\mbox{ is a periodic center leaf}\},
$$
 $$
 \hat{P}^{c}(f)=\{W^{c}(p)\subset M : W^{c}(p)\mbox{ is a periodic center leaf}\}
 $$
 and
  $$
 \hat{P}_{n}^{c}(f)=\{W^{c}(p)\in \hat{P}^{c}(f) : W^{c}(p)\mbox{ is of periodic }n\},
  $$
where $n\in\mathbb{N}$.

For any set $A\subset M$, denote $A^{c}=\bigcup_{x\in A}W^c(x)$. For any set $A$ which consists of whole center leaves, i.e., $A=A^{c}$, we call it a \emph{center set} or a $\mathcal{W}^c$-\emph{saturated set}. We say that a center leaf $W^{c}(x)$ is \emph{center nonwandering} if for any open center set $U$ which contains $W^{c}(x)$, there is $n\geq1$ such that $f^{n}U\cap U\neq\emptyset$. We denote the \emph{center nonwandering set} of $f$ by
$$
\Omega^{c}(f)=\{x\in M: W^{c}(x) \mbox{ is center nonwandering}\}.
$$
It is easy to see that $\Omega^{c}(f)$ is a closed invariant center set.
Also we denote by $\Omega(f)$ the nonwandering set of $f$. Clearly, $\Omega(f)\subset\Omega^{c}(f)$.

\begin{remark}
Generally, $\Omega(f)$ and $\Omega^{c}(f)$ for a partially hyperbolic diffeomorphism $f$ are different. For example, let $f_{1}$ be an Anosov diffeomorphism on $M$ and $f_{2}$ the standard north-south map$($see, for example, section 5.3 of \cite{Walters}$)$ on the unit circle $S^{1}$ with $\Omega(f_{2})=\{N,S\}$. We can choose $f_{1}$ and $f_{2}$ appropriately such that the direct product $f=f_{1}\times f_{2}$ is a partially hyperbolic diffeomorphism on $M\times S^{1}$. We can see that $\Omega(f)=M\times\{N,S\}$ and $\Omega^{c}(f)=M\times S^{1}$.
\end{remark}

By the above Quasi-shadowing Lemma, we have the following Center Closing Lemma.

\begin{CCL}[\cite{Hu}]\label{CCL}
Let $f$ be a partially hyperbolic diffeomorphism on $M$ with a uniformly compact center foliation.
Then for any $\varepsilon>0,$ there exists $\delta\in(0,\varepsilon)$ such that for any $x\in M$ and $n\in\mathbb{N}$ with $d(x,f^{n}x)<\delta,$
there exists a periodic center leaf $W^{c}(p)$ of period $n$ satisfying
$d_H(W^c(p), W^c(x))<\varepsilon$.

Moreover, if $d_{H}(W^{c}(x),f^{n}(W^{c}(x)))<\delta,$ then there exists a periodic center
leaf $W^{c}(p)$ of period $n$ such that $d_H(W^c(p), W^c(x))<\varepsilon.$
\end{CCL}

 An invariant center set $A$ is said to be {\em center topologically transitive},
if for any two nonempty open center sets $U,V$ in $A$, there is $n\in \mathbb{N}$
such that
$$
f^n(U)\cap V\neq \emptyset.
$$
 An invariant center set $A$ is said to be {\em center topologically mixing},
if for any two nonempty open center sets $U,V$ in $A$,
there is $n_0\in \mathbb{N}$ such that
$$
f^n(U)\cap V\neq \emptyset, \qquad \forall n\geq n_0.
$$

\begin{CSD}[\cite{Hu}]\label{CSD}
Let $f:M\to M$ be a partially hyperbolic diffeomorphism
with a uniformly compact center foliation.
Then $\Omega^c(f)$ is a union of finite pairwise disjoint closed center sets
$$
\Omega^c(f)=\Omega^c_1\cup\cdots\cup\Omega^c_k.
$$
Moreover, for each $i=1,2, \cdots, k$, $\Omega^c_i$ satisfies that

\begin{enumerate}
\item[(a)] $f(\Omega^c_i)=\Omega^c_i$ and $f|_{\Omega^c_i}$
is center topologically transitive;

\item[(b)] $\Omega^c_i=X_{1,i}\cup \cdots \cup X_{n_i,i}$ such that the $X_{j,i}$ are disjoint closed center sets, $f(X_{j,i})=X_{j+1,i}$ for $1\leq j\leq n_i-1$, $f(X_{n_i,i})=X_{1,i}$,
and $f^{n_i}|_{X_{j,i}}$ is center topologically mixing.
\end{enumerate}
We call $\Omega^{c}_{i}, i=1,2,\cdots,k$, the \emph{center basic sets} of $f$.
\end{CSD}

\section{Center specification for compact locally maximal invariant center sets}

The notion of the specification property due to Bowen (\cite{Bowen}) has turned out to be a very
important notion in the study of ergodic theory of dynamical systems. We will generalize the specification property to the partially hyperbolic systems with uniformly compact center foliations in this section and will apply this property to investigate the entropy in the next section.

Let $f$ be a partially hyperbolic diffeomorphism on $M$  with a uniformly compact center foliation. An invariant center set $A$ is called \emph{locally maximal} if there is an open center set $V$ which contains $A$ such that $A=\bigcap^{\infty}_{-\infty}f^{n}(V)$. The main task of this section is to show if a compact locally maximal invariant center set $\Lambda$ is center topologically mixing, then $f|_{\Lambda}$ has the center specification property.

 Let $A$ be an invariant set of $f$. The following definition about specification derives from section 18.3 of \cite{Katok}. A \emph{specification} $S=(\tau,P)$ of $f$ on $A$ consists of a finite collection $\tau=\{I_{1},\ldots,I_{m}\}$ of finite intervals
$I_{i}=[a_{i},b_{i}]\subset\mathbb{Z}$ and a map $P:T(\tau):=\bigcup_{i=1}^{m}I_{i}\rightarrow A$
such that for each $I\in\tau$ and for all $t_{1},t_{2}\in I$ we have $f^{t_{2}-t_{1}}(P(t_{1}))=P(t_{2})$. The specification $S$ is said
to be $n$-\emph{spaced} if $a_{i+1}>b_{i}+n$ for all $i\in\{1,\ldots,m-1\}$ and the minimal such $n$ is called
the\emph{ spacing} of $S.$ We say that $S$ \emph{parameterizes} the collection $\{P_{I}|I\in\tau\}$ of orbit segments of $f$ on $A$. Thus a specification of $f$ on $A$ is a parameterized union of orbit segments $P_{I_{i}}$ of $f$ on $A$.
We let $T(S):=T(\tau)$ and $L(S):=L(\tau):=b_{m}-a_{1}$.
The specification $S$ is called \emph{center $\varepsilon$-shadowed} by $x\in M$ if
$$
d_H(f^{n}(W^{c}(x)),W^{c}(P(n)))\le\varepsilon,\;\;\forall n\in T(S).
$$

\begin{definition}
Let $A$ be an invariant set of $f$. The map $f$ is said to have the \emph{center specification property} on $A$ if for any $\varepsilon>0$ there exists an $N=N(\varepsilon)\in\mathbb{N}$ such that any $N$-spaced specification $S$ is center $\varepsilon$-shadowed by some $x\in A$, and such that moreover for
any integer $q\geq N+L(S)$ there is a period-$q$ point $p$ center $\varepsilon$-shadowing $S$.
\end{definition}

From Lemma 6.1 of \cite{Hu}, we have the following local product structure for $f$.

\begin{lemma}\label{localproduct}
Let $f$ be a partially hyperbolic diffeomorphism on $M$ with a uniformly compact center foliation. Then there exists $\varepsilon^*>0$ such that for any $0<\varepsilon<\varepsilon^*$ there exists $\delta=\delta(\varepsilon)>0$ satisfying the following property: for any $x, y\in M$ with
$d(x, y)<\delta$, if $x_1\in W^c(x)$ then there is $y_1\in W^c(y)$
such that $W^s_{\varepsilon}(x_1)\cap W^u_{\varepsilon}(y_1)$ contains exactly one point.
\end{lemma}

\begin{lemma}\label{dense}
Let $f$ and $\Lambda$ be as in Theorem \ref{centerspecification}. Then the periodic center leaves are dense in $\Lambda$, i.e., $\overline{P^{c}(f)|_{\Lambda}}=\Lambda$. Moreover, for any periodic center leaf $W^c(p)$ in $\Lambda$, we have  $W^{cu}(p)$ is dense in $\Lambda$, i.e., $\overline{W^{cu}(p)}\supset\Lambda$.
\end{lemma}

\begin{proof}
Since $\Lambda$ is center topologically mixing, by the Center Closing Lemma, we obtain that for each $x\in \Lambda$ and any $\varepsilon>0$ there exists a periodic center leaf $W^c(p)$ such that $d_H(W^c(x), W^c(p))\le \varepsilon$. Therefore, any center leaf which lies in $\Lambda$ can be approximated by a sequence of  periodic center leaves, and hence the periodic center leaves are dense in $\Lambda$.

 Given a periodic center leaf $W^c(p)$ in $\Lambda$. We will prove $W^{cu}(p)$ is dense in $\Lambda$. Since $\overline{P^{c}(f)|_{\Lambda}}=\Lambda$,  we only need to prove that for any $q$ which lies in a periodic center leaf and any $\varepsilon>0$ we can find a point $p_0\in W^{cu}(p)$ such that $d(p_0, q)<\varepsilon$.
 Since the center foliation is uniformly compact and $\Lambda$ is center topologically mixing, for any open center sets $U$ and $V$ which contains $W^c(p)$ and $W^c(q)$ respectively, there exists $n_{0}\in\mathbb{N}$ such that $f^{n}(U)\cap V\neq\emptyset$ for any $n\geq n_{0}$. For the above $\varepsilon$, we take $\delta>0$ such that the corresponding result in Lemma \ref{localproduct} holds. Since the center foliation is continuous, we can require that $U$ and $V$ are both ``thin" enough such that
  \begin{equation}\label{Uinball}
  U\subset B_{d_H}\big(W^c(p),\delta\big)\;\;\text{and}\;\;V\subset B_{d_H}\big(W^c(q),\varepsilon\big).
  \end{equation}
Suppose $W^{c}(p)$ has period $m$, we take a number $n\in\mathbb{N}$ such that $mn\geq n_{0}$ and $f^{mn}(U)\cap V\neq\emptyset$. Hence we can choose a point $x\in \Lambda$ such that $W^{c}(x)\subset U$
and $f^{mn}(W^{c}(x))=W^{c}(f^{mn}(x))\subset V$. By (\ref{Uinball}), we can take a point $y\in W^{c}(f^{mn}(x))$ such that $d(y,q)<\varepsilon$.  Note that
$d_{H}(W^{c}(p),W^{c}(x))<\delta<\varepsilon$ and $f^{-mn}(y)$ lies in $W^{c}(x)$. If $W^c(x)=W^c(p)$ then $y\in W^{c}(f^{mn}(x))=W^{c}(f^{mn}(p))=W^{c}(p)$. Hence $p_0=y$ is a desired point. Now suppose $W^c(x)\neq W^c(p)$. Take $p'\in W^c(p)$ such that $d(p',f^{-mn}(y))<\delta$.
By Lemma \ref{localproduct}, $W^u_{\varepsilon}(p')\cap W^{s}_{\varepsilon}(f^{-mn}(y))$ contains exactly one point, say $p''$.  Then $p_{0}=f^{mn}(p'')$ is a desired point.
 \end{proof}

In the above lemma, we have shown that for any point lying in a periodic center leaf in $\Lambda$, its center-unstable manifold is dense in $\Lambda$.  Since the periodic center leaves are dense in $\Lambda$, we get that for each point in $\Lambda$, its center-unstable manifold is dense in $\Lambda$. Furthermore, the following lemma tells us that the above density is even uniform.

\begin{lemma}\label{uniformlydense}
Let $f$ and $\Lambda$ be as in Theorem \ref{centerspecification}. Then for any $\alpha>0$ there is $N\in\mathbb{N}$ such that for any $x,y\in\Lambda$ and $n\geq N$ we have $f^{n}(W^{c}(W^{u}_{\alpha}(x)))\cap W^{s}_{\alpha}(y)\neq\emptyset$, where $W^{c}(W^{u}_{\alpha}(x))=\cup_{x'\in W^{u}_{\alpha}(x)}W^{c}(x',f)$.
\end{lemma}

\begin{proof}
By Lemma \ref{localproduct}, we have that there is a $\varepsilon^{*}>0$ such that for any $0<\varepsilon<\varepsilon^*$ there exists $\delta=\delta(\varepsilon)>0$ satisfies the following property: for any $x, y\in M$ with
$d(x, y)<\delta$, if $x_1\in W^c(x)$ then there is $y_1\in W^c(y)$
such that $W^s_{\varepsilon}(x_1)\cap W^u_{\varepsilon}(y_1)$ contains exactly one point.

Given $\alpha>0$. Since the center foliation is continuous and uniformly compact, there exists $0<\eta<\min\{\varepsilon^{*}, \alpha/2, \delta(\alpha/2)/3\}$ such that
\begin{equation}\label{etaH}
 d(x,y)<\eta\Longrightarrow d_H(W^c(x), W^c(y))<\delta(\alpha/2)/3.
\end{equation}
To choose $N$ take a $\eta$-dense set $\{p_k : k=1,\ldots,r\}$ of points each of which lies in a periodic center leaf (with period $t_{k}$).
By Lemma \ref{dense}, for each $1\le k\le r$, $W^{cu}(p_k)$ is dense in $\Lambda$, and hence there exists $m_k$ such that $f^{mt_{k}}(W^{c}(W^{u}_{\eta}(p_{k})))$ is $\delta(\eta)$-dense for all $m\geq m_{k}$, that is, for any $x'\in \Lambda$ there is $x''\in f^{mt_{k}}(W^{c}(W^{u}_{\eta}(p_{k})))$ such that $d(x', x'')<\delta(\eta)$. Let $N=\prod^{r}_{k=1}m_{k}t_{k}$ and note that $f^{N}(W^{c}(W^{u}_{\eta}(p_{k})))$ is $\delta(\eta)$-dense for all $k$.

Now we show that $N$ is as desired. For $x,y\in\Lambda$ take $j$ such that $d(x,p_{j})<\eta$ (hence by (\ref{etaH}), $d_H(W^c(x), W^c(p_j))<\delta(\alpha/2)/3$ ), $z\in f^{N}(W^{c}(W^{u}_{\eta}(p_{j})))$ such that $d(y,z)\leq\delta(\eta)$,  and $w\in W^{u}_{\eta}(z)\cap W^{s}_{\eta}(y)$.  Clearly,
\begin{eqnarray*}
d(f^{-N}(w),W^{c}(p_j))&\le& d(f^{-N}(w), f^{-N}(z))
+d(f^{-N}(z),W^{c}(p_j))\\
 &\le& \eta+\delta(\alpha/2)/3<2\delta(\alpha/2)/3.
\end{eqnarray*}
Therefore,
by triangle inequality we have
$$
d(f^{-N}(w),W^{c}(x))\leq d(f^{-N}(w),W^{c}(p_j))+d_H(W^c(p_j), W^c(x))<\delta(\alpha/2).
 $$
 So there exists $v\in W^{c}(W^{u}_{\alpha/2}(x))\cap W^{s}_{\alpha/2}(f^{-N}(w))$ by Lemma \ref{localproduct} and $$
 f^{N}(v)\in f^{N}(W^{c}(W^{u}_{\alpha/2}(x)))\cap W^{s}_{\alpha/2}(w)\subset f^{N}(W^{c}(W^{u}_{\alpha}(x)))\cap W^{s}_{\alpha}(y)\neq\emptyset
 $$
 since $w\in W^{s}_{\eta}(y)$ and $\eta\le\alpha/2$. For $x,y\in\Lambda, n\geq N$ note that
 $$f^{n}(W^{c}(W^{u}_{\alpha}(x)))\cap W^{s}_{\alpha}(y)\supset f^{N}(W^{c}(W^{u}_{\alpha}(f^{n-N}(x))))\cap W^{s}_{\alpha}(y)\neq\emptyset.
 $$
This completes the proof of this lemma.
\end{proof}

\begin{proof}[Proof of Theorem \ref{centerspecification}]
Given $0<\varepsilon <\varepsilon^*$ (recall that $\varepsilon^*$ comes from Lemma \ref{localproduct}). We assume that the local stable and unstable manifolds satisfy that
\begin{equation}\label{condlambda}
x\in W^s_{\eta}(y)\Longrightarrow f(x)\in W^s_{\lambda\eta}(f(y))\;\mbox{ and }\;x'\in W^u_{\eta}(y')\Longrightarrow f^{-1}(x')\in W^u_{\lambda\eta}(f^{-1}(y'))
\end{equation}
for any $0<\eta<\varepsilon$.
 Take $\alpha<\varepsilon/2$  such that
\begin{equation}\label{alphaH}
 d(x,y)<2\alpha\Longrightarrow d_H(W^c(x), W^c(y))<\varepsilon/2.
\end{equation}
For this $\alpha$, take the corresponding $N$ obtained in Lemma \ref{uniformlydense}. Increase $N$ if necessary such that
\begin{equation}\label{bigN}
\lambda^{N}<1/2,
\end{equation}
 where
$\lambda$ is the contraction rate in the definition of partially hyperbolicity.

Given an $N$-spaced specification $S=(\tau,P)$ of $f$ on $\Lambda$  in which $\tau=\{I_{k}=[a_k,b_k]: 1\le k\le m\}$ and $P$  is the corresponding map on $T(\tau)$.
We let $x_{1}=f^{-a_{1}}(P(a_{1}))$ and define $x_{2},x_{3},\ldots,x_{m}$ as follows: Given $x_{k}$ there exists, by Lemma \ref{uniformlydense}, an $x_{k+1}$ such that
$$
f^{a_{k+1}}(x_{k+1})\in f^{a_{k+1}-b_{k}}(W^{c}(W^{u}_{\alpha}(f^{b_{k}}(x_{k}))))\cap W^{s}_{\alpha}(P(a_{k+1}))
$$
since $a_{k+1}-b_{k}\geq N$ by assumption.

In the following, we will show that $S$ can be center $\varepsilon$-shadowed by $x :=x_{m}$. Since for each $k\in[1,m]$, $f^{a_{k}}(x_{k})\in W^{s}_{\alpha}(P(a_{k}))$ by construction,
$$
d(f^{n}(x_{k}),P(n))=d(f^{n-a_{k}}(f^{a_{k}}(x_{k})),f^{n-a_{k}}(P(a_{k})))\leq\alpha,
$$
and hence by (\ref{alphaH}) we have
$$
d_{H}(W^{c}(f^{n}(x_{k})),W^{c}(P(n)))\leq  \varepsilon/2
$$
for $n\in I_{k}$.
Once we prove
\begin{equation}\label{shadow1}
d_{H}(W^{c}(f^{n}(x)),W^{c}(f^{n}(x_{k})))<\varepsilon/2
\end{equation}
for all $n\in I_{k}, k\in[1,m]$, then we get the desired result by the triangle inequality.

Now we show (\ref{shadow1}). For $k=m$ and $n\in I_m$,
(\ref{shadow1}) holds obviously since $x=x_m$. So we begin at $k=m-1$ and $n\in I_{m-1}$.
Since $f^{b_{m-1}}(x)\in W^{c}(W^{u}_{\alpha}(f^{b_{m-1}}(x_{m-1})))$ by construction, we can take $x'_{m-1}\in W^c(x)$
such that
$$f^{b_{m-1}}(x'_{m-1})\in W^{u}_{\alpha}(f^{b_{m-1}}(x_{m-1}))\cap W^{c}(f^{b_{m-1}}(x)),$$
then by (\ref{alphaH}) we have
$$
d_{H}(W^{c}(f^{b_{m-1}}(x'_{m-1})),W^{c}(f^{b_{m-1}}(x_{m-1})))<\varepsilon/2,
$$
and hence by (\ref{condlambda}), we have
\begin{equation}\label{m-1}
d_{H}(W^{c}(f^{n}(x'_{m-1})),W^{c}(f^{n}(x_{m-1})))<\varepsilon/2
\end{equation}
for $n\in I_{m-1}$ and
$$
f^{a_{m-1}}(x'_{m-1})\in W^{u}_{\alpha\lambda^{b_{m-1}-a_{m-1}}}(f^{a_{m-1}}(x_{m-1}))\cap W^{c}(f^{a_{m-1}}(x)).
$$
Now consider $k=m-2$ and $n\in I_{m-2}$. Note that
$$f^{b_{m-2}}(x'_{m-1})\in W^{c}(W^{u}_{\alpha\lambda^{b_{m-1}-b_{m-2}}}(f^{b_{m-2}}(x_{m-1})))$$
by (\ref{condlambda}), and $f^{b_{m-2}}(x_{m-1})\in W^{c}(W^{u}_{\alpha}(f^{b_{m-2}}(x_{m-2})))$, we have
$$
f^{b_{m-2}}(x'_{m-1})\in W^{c}(W^{u}_{\alpha+\alpha\lambda^{b_{m-1}-b_{m-2}}}(f^{b_{m-2}}(x_{m-2}))).
$$
Take $x'_{m-2}\in W^c(x'_{m-1})$  such that
$$f^{b_{m-2}}(x'_{m-2})\in W^{u}_{\alpha+\alpha\lambda^{b_{m-1}-b_{m-2}}}(f^{b_{m-2}}(x_{m-2}))\cap W^{c}(f^{b_{m-2}}(x'_{m-1})),$$
 then by (\ref{alphaH}) and (\ref{bigN}) (note that $b_{m-1}-b_{m-2}>N$) we have
$$
d_{H}(W^{c}(f^{b_{m-2}}(x'_{m-2})),W^{c}(f^{b_{m-2}}(x_{m-2})))<\varepsilon/2.
$$
and hence by (\ref{condlambda}), we have
\begin{equation}\label{m-2}
d_{H}(W^{c}(f^{n}(x'_{m-2})),W^{c}(f^{n}(x_{m-2})))<\varepsilon/2
\end{equation}
for $n\in I_{m-2}$.
Inductively, we get $x'_{1}, x'_{2}, \ldots, x'_{m-3}$ such that for each $3\le i\le m-1$, we have
$x'_{m-i}\in W^c(x'_{m-i+1})$  such that
$$
f^{b_{m-i}}(x'_{m-i})\in W^{u}_{\alpha(1+\sum_{j=1}^{i-1}\lambda^{b_{m-j}-b_{m-i}})}(f^{b_{m-i}}(x_{m-i}))\cap W^{c}(f^{b_{m-i}}(x'_{m-i+1}))
$$
and
\begin{equation}\label{m-i}
d_{H}(W^{c}(f^{n}(x'_{m-i})),W^{c}(f^{n}(x_{m-i})))<\varepsilon/2.
\end{equation}
for $n\in I_{m-i}$.
Note that
$$
W^{c}(x'_{1})=W^{c}(x'_{2})=\cdots=W^{c}(x'_{m-1})=W^{c}(x),
$$
by (\ref{m-1}), (\ref{m-2}) and (\ref{m-i}), we have $d_{H}(W^{c}(f^{n}(x)),W^{c}(f^{n}(x_{k})))<\varepsilon/2$
for all $n\in I_{k}, k\in[1,m]$, i.e., (\ref{shadow1}) holds.

Now we show that for any $q\geq N+L(S)$ there is a period-$q$ center leaf in $\Lambda$ center $\varepsilon$-shadowed $S$.
By the Center Closing Lemma, there exists $0<\delta<\varepsilon/2$ such that for any $x\in M$ and $n\in\mathbb{N}$ with $d_H(W^c(x),W^c(f^{n}(x)))<\delta$,
there exists a periodic center leaf $W^{c}(p)$ of period $n$ satisfying
\begin{equation}\label{shadow2}
d_H(W^c(f^{i}(p)),W^c(f^{i}(x)))\leq\varepsilon/2,\;\;\;\forall\;1\le i\le n.
\end{equation}
As we discussed above, we increase $N$ if necessary such that any $N$-spaced specification can be center $\delta/2$-shadowed by some point. We define another specification $S'=(\tau',P')$ with
$\tau'=\tau\cup\{\{a_{1}+q\}\}$ and $P'|_{T(\tau)}=P, P'(a_{1}+q)=P(a_{1}),$ which is clearly
$N$-spaced. We thus obtain a point $x' :=f^{a_{1}}(x)\in \Lambda$ such that
\begin{eqnarray*}
&&d_{H}(W^{c}(x'),W^{c}(f^{q}(x')))\\
&\leq& d_{H}(W^{c}(x'),W^{c}(P(a_{1})))+d_{H}(W^{c}(f^{q}(x')),W^{c}(P(a_{1})))\leq\delta.
\end{eqnarray*}
Therefore, by (\ref{shadow2}), there exists a periodic center leaf $W^{c}(p)$ of period $q$ such that
$$
d_{H}(W^{c}(f^{n+a_{1}}(p)),W^{c}(f^{n}(x')))\leq\varepsilon/2
$$
 for all $n\in [0,q]$.
Since $\Lambda$ is locally maximal, we get $W^{c}(p)\subset\Lambda$.  By the triangle inequality, the specification $S$ is center $\varepsilon$-shadowed the period-$q$ center leaf $W^{c}(p)$.

This completes the proof of this theorem.
\end{proof}

\section{Entropy and Periodic Center Leaves}

In this section we shall use the center specification property to investigate the relationship among topological entropy, entropy of the center foliation and the growth rate of periodic center leaves for partially hyperbolic diffeomorphisms with a uniformly compact center foliation.

Let $f:M\rightarrow M$ be a partially hyperbolic diffeomorphism. Denote
$$
p^{c}(f) =\limsup_{n\rightarrow\infty}\frac{1}{n}\log\mbox{card}(\hat{P}^{c}_{n}(f))
$$
and
$$
h(f,\mathcal{W}^{c})=\lim_{\varepsilon\longrightarrow0}\limsup_{n\rightarrow\infty}\frac{1}{n}\log s(\varepsilon, n, W^{c}, f),
$$
where $s(\varepsilon, n, W^{c}, f)=\max_{x\in M}s(\varepsilon, n, W^{c}(x), f)$ in which $s(\varepsilon, n, W^{c}(x), f)$ is the largest cardinality of
any $(\varepsilon, n, f)$-separated set of $W^{c}(x)$ (See, for example, section 7.2 of \cite{Walters} for the precise definition of $(\varepsilon, n, f)$-separated set).

\begin{proposition}\label{expansivecase}
Let $f:M\rightarrow M$ be a partially hyperbolic diffeomorphism. If $f$ is plaque expansive, then
\begin{equation}\label{entropyinequality1}
p^{c}(f)\leq h(f).
\end{equation}
\end{proposition}

\begin{proof}
Assume $f$ is plaque expansive with the plaque expansiveness constant $\eta$. Given $n\in \mathbb{N}$. Let
$\mathcal{F}=\{W^c(p_i):f^n(W^c(p_i))=W^c(p_i), i\in I\}$ be the family of pairwise different period-$n$ center leaves.  Since $\eta$ is the plaque expansiveness constant of $f$, the set $A=\{p_i: i\in I\}$ is an $(\eta, n, f)$-separated set. Therefore, card$(\hat{P}^{c}_{n}(f))\le s(\varepsilon, n, f)$ for any $0<\varepsilon<\eta$, and hence (\ref{entropyinequality1}) holds.
\end{proof}

Denote $\mathcal{M}(M,f)$ the space of invariant measures of $f$ with weak$^*$ topology.

\begin{lemma}\label{support}
Let $f$ be as in Theorem \ref{entropyrelation}. For any $\mu\in \mathcal{M}(X,f)$, we have
$$
\mbox{\emph{supp}}\mu\subset\overline{P^{c}(f)}.
$$
\end{lemma}
\begin{proof}
Let $x_{0}\in$ supp$\mu$ and fix $\varepsilon>0.$ We can make $\mu(B(x_{0},\varepsilon)\cap M)>0.$ For $\varepsilon>0$, choose $0<\delta<\varepsilon$ as in Center Closing Lemma.

Now pick a set $B\subset B(x_{0},\varepsilon)\cap M$ of diameter less than $\delta$ and of positive measure. By the Poincar$\acute{e}$ Recurrence Theorem, for almost every $x\in B$ there exists a positive integer $n(x)$ such that $f^{n(x)}(x)\in B,$ and consequently $d(x,f^{n(x)}(x))<\delta.$ Applying the Center Closing Lemma then we obtain that there exists a periodic center leaf $W^{c}(p)$ of period $n(x)$ such that $d(p,x)\leq\varepsilon$, and clearly $d(x_{0},p)<d(x_{0},x)+d(x,p)<2\varepsilon$. Since $\varepsilon$ is taken arbitrarily, we can see that $x_0\in\overline{P^{c}(f)}$, and hence the desired result holds.
\end{proof}

\begin{proof}[Proof of Theorem \ref{entropyrelation}]
By the Variational Principle (see Theorem 8.6 of \cite{Walters}),
 $$
 h(f)=\sup\{h_{\mu}(f): \mu\in \mathcal{M}(M,f) \}.
 $$
From Theorem D of \cite{Hu}, $\overline{P^{c}(f)}=\Omega^{c}(f)$, and hence by Lemma \ref{support}, for any $\mu\in \mathcal{M}(X,f)$, supp$\mu\subset \Omega^{c}(f)$.  Therefore
\begin{eqnarray}\label{nonwand}
 h(f)&=&\sup\{h_{\mu}(f|_{\Omega^{c}(f)}): \mu\in \mathcal{M}(M,f)\}\notag\\
 &=&\sup\{h_{\mu}(f|_{\Omega^{c}(f)}): \mu\in \mathcal{M}(\Omega^{c}(f), f|_{\Omega^{c}(f)})\}\\
 &=&h(f|_{\Omega^{c}(f)}).\notag
\end{eqnarray}

The Center Spectral Decomposition Theorem tells us that there is some $N\in\mathbb{N}$ such that $\Omega^{c}(f)$
can be decomposed into finitely many components, say $X_1,\cdots,X_l$, on which $f^{N}$ is center topologically mixing.
By Proposition 3.1.7 (2) of \cite{Katok},
 $$
h(f^N|_{\Omega^{c}(f)})=\max \{h(f^N|_{X_i}): 1\le i\le l\}.
 $$
Take some $1\le i^*\le l$ such that $h(f^N|_{\Omega^{c}(f)})=h(f^N|_{X_{i^*}})$.  Once showing
\begin{equation}\label{entropyinequality3}
h(f^N|_{X_{i^*}})\le p^{c}(f^{N}|_{X_{i^*}})+h(f^N|_{X_{i^*}}, \mathcal{W}^{c}),
\end{equation}
 we obtain the desired inequality (\ref{entropyinequality}) immediately from the following observations
 $$
h(f^N|_{X_{i^*}})\stackrel{(\ref{nonwand})}{=}h(f^N|_{\Omega^{c}(f)})=h(f^N)=Nh(f),
$$
$$
 p^{c}(f^{N}|_{X_{i^*}})\le p^{c}(f^{N})\le Np^{c}(f)
$$
and
$$
h(f^N|_{X_{i^*}}, \mathcal{W}^{c})\le h(f^N, \mathcal{W}^{c})\le Nh(f, \mathcal{W}^{c}).
$$

Now we prove (\ref{entropyinequality3}). For simplicity of the notation, we let $g=f^N$. Let $E_{n}$ be an $(n,\varepsilon)$-separated set of $X_{i^*}$ with respect to $g$.   By Theorem \ref{centerspecification}, any two different elements of $E_{n}$ which belong two different center leaves can be center $\varepsilon/2$-shadowed by two different center periodic leaves of period $n+N(\varepsilon/2)$ (with respect to $g$). Assume the points in $E_{n}$ belong to $N^{*}$ different center leaves, the center leaf $W^{c}(p)$ contains the most points in $E_{n}$. Then we have
$$
\mbox{card}(E_n)\leq N^{*}\cdot s(\varepsilon, n, W^{c}(p), g|_{X_{i^*}})\leq \mbox{card}(\hat{P}^{c}_{n+N(\varepsilon/2)}(g|_{X_{i^*}}))\cdot s(\varepsilon, n, W^c, g|_{X_{i^*}}),
$$
and hence (\ref{entropyinequality3}) holds.

Now assume the center foliation $\mathcal{W}^{c}$ is of dimension one. Note that for any periodic center leaf $W^{c}(p)$ of period $n$, $f^{n}: W^{c}(p)\rightarrow W^{c}(p)$ is a homeomorphism which is topologically conjugate to a homeomorphism on the circle, and hence $h(f^{n},W^{c}(p))=0$.
Since $f^{N}|_{X_{i^*}}$ is center topologically mixing, the periodic center leaves are dense in $X_{i^*}$, we have
$h(f^N|_{X_{i^*}}, \mathcal{W}^{c})=0$, and hence $h(f)\leq p^{c}(f)$ by (\ref{entropyinequality3}). Since the center foliation of $f$ is uniformly compact, it is plaque expansive by Proposition 13 of \cite{Pugh}, and hence
$p^{c}(f)\le h(f)$ by Proposition \ref{expansivecase}. Therefore, the desired equality (\ref{entropyequality}) holds.
\end{proof}

\begin{remark}
We also notice that Carrasco \cite{Carrasco11} got a similar result for the induced map $g$ on the quotient space $M/\mathcal{W}^{c}$ of the center foliation under the  condition that the center foliation of $f$ is uniformly compact and without holonomy.
\end{remark}

\end{document}